\documentclass[10pt]{amsart}

\usepackage{amssymb,amsmath,amsthm,amsfonts}
\usepackage{pdfsync}

\newcommand{\comment}[1]{}

\newtheorem{propn}{Proposition}
\newtheorem{thm}{Theorem}
\newtheorem{cor}{Corollary}

\newtheorem*{thmA}{Theorem A}
\newtheorem*{thmB}{Theorem B}
\newtheorem*{thmC}{Theorem C}
\theoremstyle{remark}

\theoremstyle{definition}

\makeatletter
\def\imod#1{\allowbreak\mkern10mu({\operator@font mod}\,\,#1)}
\makeatother

\newcommand{\R}{\mathbb R}

\newcommand{\Z}{\mathbb Z}

\newcommand{\D}{\delta}

\newcommand{\VE}{\varepsilon}
\newcommand{\A}{\alpha}
\newcommand{\B}{\beta}

\newcommand{\be}{\begin{equation}}
\newcommand{\ee}{\end{equation}}
\newcommand{\bee}{\begin{equation*}}
\newcommand{\eee}{\end{equation*}}

\begin{document}
\title[Simultaneous Polynomial Recurrence Modulo 1]{A Purely Combinatorial Approach to\\ Simultaneous Polynomial Recurrence Modulo 1}
\author{Ernie Croot\quad\quad\quad  Neil Lyall\quad\quad\quad Alex Rice}

\address{School of Mathematics, Georgia Institute of Technology, Atlanta, GA 30332, USA}
\email{ecroot@math.gatech.edu}

\address{Department of Mathematics, The University of Georgia, Athens, GA 30602, USA}
\email{lyall@math.uga.edu}

\address{Department of Mathematics, Bucknell University, Lewisburg, PA 17837, USA}
\email{alex.rice@bucknell.edu}

\subjclass[2010]{11J54,\,11J71}

\begin{abstract}
Using purely combinatorial means we obtain results on simultaneous Diophantine approximation modulo 1 for systems of polynomials with real coefficients and no constant term.
\end{abstract}
\maketitle

\setlength{\parskip}{5pt}


\section{Introduction}

We begin by recalling the well-known Kronecker approximation theorem:
\begin{thmA}[Kronecker Approximation Theorem]
Given any real numbers $\A_1,\dots,\A_d$, and an integer $N\geq1$, there exists an integer $1\leq n\leq N$ such that \[\|n\A_j\|\ll N^{-1/d} \text{ \ for \ }j=1,\dots,d.\]
\end{thmA}

\emph{Remark on Notation:} In Theorem A above, and in the rest of this paper, we use the standard notations $\|\A\|$ to denote, for a given $\A\in\R$, the distance from  $\A$ to the nearest integer and the Vinogradov symbol $\ll$ to denote ``less than an absolute constant times''.

Kronecker's theorem is of course an almost immediate consequence of the pigeonhole principle: one simply partitions the torus $(\R/\Z)^d$ into $N$ ``boxes" of diameter $O(N^{1/d})$ and   considers the orbit of $(n\A_1,\dots, n\A_d)$. 
In \cite{GT}, Green and Tao presented a proof of the following quadratic analogue of the above theorem, due to Schmidt \cite{Schmidt}.

\begin{thmB}[Simultaneous Quadratic Recurrence, Proposition A.2 in \cite{GT}]
Given any real numbers $\A_1,\dots,\A_d$, and an integer $N\geq1$, there exists an integer $1\leq n\leq N$ such that \[\|n^2\A_j\|\ll dN^{-c/d^2} \text{ \ for \ }j=1,\dots,d.\]
\end{thmB}

The argument presented by Green and Tao in \cite{GT} was later extended (in a straightforward manner) by the second author and Magyar in \cite{LM3} to any system of polynomials without constant term.

\begin{thmC}[Simultaneous Polynomial Recurrence, consequence of Proposition B.2 in \cite{LM3}]
Given any system of polynomials $h_1,\dots,h_d$ of degree at most $k$ with real coefficients and no constant term and an integer $N\geq1$, there exists an integer $1\leq n\leq N$ such that \[\|h_j(n)\|\ll k^2d N^{-c_k/d^2} \text{ \ for \ }j=1,\dots,d.\]\end{thmC}

To the best of our knowledge the only known proofs of Theorems B and C, with bounds of comparable strength to those stated above, all follow the general framework of Schmidt's original argument (although it should be noted that he in fact only worked with quadratic polynomials). This argument is Fourier analytic in nature and relies on estimates for exponential sums (Weyl's inequality) as well as an induction on $d$; we refer the reader to \cite{GT} for a more comprehensive sketch of the main features of Schmidt's argument. It would be of interest to find a new approach to this problem that gives bounds of the form $N^{-c_k/d^{C_k}}$ by more elementary means.

The purpose of this note is to present an elementary approach to simultaneous polynomial recurrence modulo 1, for systems of polynomials with real coefficients and no constant term, that gives effective quantitative bounds of the form 
\[\exp\bigl(-c_1 d^{-C}(\log N)^{c_2}\bigr)\]
for certain small constants $c_1,\,c_2>0$ and large constant $C>0$ that depend on $k$, but are independent of $d$.

The precise statements of the main results of this paper, namely Theorems \ref{1} and \ref{3}, can be found at the beginning of Sections \ref{Sec2} and \ref{Sec3} respectively. Theorem \ref{1} covers the special case of quadratic polynomials and its proof is presented in Section \ref{Sec2} in isolation from the general argument 
as it is in this case that our argument is at its most transparent. The proof of Theorem \ref{3}, which covers the general case for polynomials of arbitrary degree, is presented in Section \ref{Sec3} and proceeds via a somewhat intricate (double) inductive argument on the maximum degree of the polynomials.


\section{Simultaneous Quadratic Recurrence}\label{Sec2}
In this section we present an elementary approach to simultaneous quadratic recurrence modulo 1. 
Our main result is the following

\begin{thm}\label{1}
Given any real numbers $\A_1,\dots,\A_d$, and an integer $N\geq1$, there exists an integer $1\leq n\leq N$ and small absolute constant $c_1>0$ such that \[\|n^2\A_j\|\leq \exp\bigl(-c_1(d^{-10}\log N)^{1/4}\bigr) \text{ \ for \ }j=1,\dots,d.\]
\end{thm}

\subsection{An elementary consequence of Theorem \ref{1} and Kronecker's theorem}
Before embarking on the proof of Theorem \ref{1} we observe that by combining Theorem \ref{1} with Kronecker's theorem we can obtain the following simultaneous recurrence result for any collection of quadratic polynomials with real coefficients and no constant term. In addition to being of interest in its own right (and extremely straightforward), this will serve as a model for the inductive step in  the proof of Theorem \ref{3} in Section \ref{Sec3} below.

\begin{cor}\label{cor}
Given any system of \emph{quadratic} polynomials $h_1,\dots,h_d$ with real coefficients and no constant term and an integer $N\geq1$, there exist an integer $1\leq n\leq N$ and a small absolute constant $c_2>0$ such that
\[\|h_j(n)\|\ll \exp\bigl(-c_2(d^{-14}\log N)^{1/4}\bigr) \text{ \ for \ }j=1,\dots,d.\]
\end{cor}

\begin{proof}
Let $N\geq 1$ be an integer and $h_1,\dots,h_d$ be a collection of polynomials of the form
\[h_j(x)=x^2\A_j+x\B_j\]
with $\A_j$ and $\B_j$ real for each $j=1,\dots,d$. We now set
\[N':=\exp\bigl(-c_2(d^{-10}\log N)^{1/4}\bigr)\]
where $c_2=\min\{c_1/8,1/16\}$ and observe that $N'\leq N^{1/16}$.

It follows from Theorem \ref{1} that there exists an integer $1\leq n\leq N/N'$ such that
\[\|n^2\A_j\|\leq \exp\bigl(-c_1(d^{-10}\log N/N')^{1/4}\bigr)\leq\exp\bigl(-\dfrac{c_1}{2}(d^{-10}\log N)^{1/4}\bigr)\]
for $j=1,\dots,d$, where in the second inequality we have simply used the fact that $N'\leq N^{1/16}$. It thus follows that for \underline{all} $1\leq t\leq N'$ we have
\be\label{from1}\|(tn)^2\A_j\|\leq N'^2\exp\bigl(-\dfrac{c_1}{2}(d^{-10}\log N)^{1/4}\bigr)\leq \exp\bigl(-\dfrac{c_1}{4}(d^{-10}\log N)^{1/4}\bigr)\ee
for $j=1,\dots,d$.
We finish by noting that Kronecker's theorem gives the existence of an integer $1\leq t\leq N'$ such that
\be\label{fromK}\|t(n\B_j)\|\ll N'^{-1/d}\leq \exp\bigl(-c_2\,d^{-1}(d^{-10}\log N)^{1/4}\bigr)\ee
for $j=1,\dots,d$. The result follows.
\end{proof}

\subsection{Proof of Theorem \ref{1}}
We begin with a definition. 

For a given real parameter $\VE>0$ and positive integers $d$, $s$, and $M$, we let $P_d(s,M,\VE)$ denote the statement:
\begin{quote}
\emph{Given any real numbers $\A_1,\dots,\A_d$, there exists integers \mbox{$0\leq n_1,\dots,n_s\leq M$} (not all equal to zero) such that \[\|(n_1^2+\cdots+n_s^2)\A_j\|\leq\VE  \text{ \ for \ }j=1,\dots,d.\]}
\end{quote}

The key proposition is the following
\begin{propn}\label{Propn1}
Given any real parameter $\VE>0$ and positive integers $d$, $s$, and $M$, then
\[P_d(2s,M,\VE) \ \Longrightarrow \ P_d(s,N,2\VE)\]
for any integer $N\geq \exp\bigl(C_0d^2\log^2(\VE^{-1}sM^2)\bigr)$, where $C_0\geq1$ is some absolute constant.
\end{propn}

\begin{proof}[Proof that Proposition \ref{Propn1} implies Theorem \ref{1}]

It follows immediately from Kronecker's theorem and Lagrange's four squares theorem that the statement $P_d(4, \VE^{-d/2}, \VE)$ holds. By applying Propostion \ref{Propn1} two times we can thus conclude that the statement 
\[P_d(1,\underbrace{\exp\bigl(C_0d^2\log^2\bigl(\VE^{-1}\exp(C_0d^2\log^2(\VE^{-(d+1)}))\bigr)\bigr)}_{(\star)}, 4\VE)\]
also holds. Since $(\star)$ is easily dominated by $\exp(Cd^{10}\log^4\VE^{-1})$ with  $C=16C_0(C_0+1)^2$, Theorem \ref{1} follows as stated by setting $N=\exp(Cd^{10}\log^4\VE^{-1})$ with $C>0$ sufficiently large, and solving for $\VE$.
\end{proof}

\subsection{Proof of Proposition \ref{Propn1}}

Let $\A_1,\dots,\A_d$ be any given real numbers and $(m,X)$ be a pair of positive integer parameters to be determined. 

It follows from Kronecker's theorem that there exist integers $1\leq x_1,\dots,x_m\leq X$ such that
\be\label{est1}\|x_{i}x_{i'}\A_j\|\ll X^{-1/d(m-1)}\text{ \ for \ }1\leq i<i'\leq m \text{ and } j=1,\dots,d.\ee

We now consider all subset-sums of the $x_i^2$. It follows from the pigeonhole principle that there exist two of these sums
\[Y:=y_1^2+\cdots+y_Q^2\quad\text{and}\quad Z:=z_1^2+\cdots+z_R^2\]
with $y_q\neq z_r$ for all $1\leq q\leq Q$ and $1\leq r\leq R$, and real numbers $\B_1,\dots,\B_d$ such that for each $j=1,\dots,d$, the sums of squares $Y$ and $Z$ each satisfy 
\[Y\A_j\equiv\B_j+\D_{Y,j} \imod{1} \quad\text{and}\quad Z\A_j\equiv\B_j+\D_{Z,j} \imod{1}\]
with both $|\D_{Y,j}|\leq 2^{-m/d}$ and $|\D_{Z,j}|\leq 2^{-m/d}$.
Now the assumed validity of the statement $P_d(2s,M,\VE)$ allows us to conclude that  there must exist integers $1\leq n_1,\dots,n_{2s}\leq M$ such that for each $j=1,\dots,d$ 
\[\|(n_1^2+\cdots+n_{2s}^2)\B_j\|\leq\VE\]
and hence
\be\label{est2}\|\{Y(n_1^2+\cdots+n_s^2)+Z(n_{s+1}^2+\cdots+n_{2s}^2)\}\A_j\|\leq\VE+2s M^2 2^{-m/d}.\ee

Since $Y$ and $Z$ are themselves sums of squares (in fact sums of different $x_i^2$) and \[X_1^2+\cdots+X_n^2=(X_1+\cdots+X_n)^2 - \sum_{\substack{i,i'=1\\i\ne i'}}^n X_iX_{i'},\]
the only properties of squares that we are using in this argument, it follows from (\ref{est1}) and (\ref{est2}) that
\[\left\|\left(\sum_{i=1}^s \Bigl\{n_i(y_1+\cdots+y_Q)+n_{s+i}(z_1+\cdots+z_R)\Bigr\}^2\right)\A_j\right\|\leq \VE+2s M^2 2^{-m/d}+O\bigl(sm^2 M^2 X^{-1/d(m-1)}\bigr)\]
for each $j=1,\dots,d$. Note also that 
\[\bigl|n_i(y_1+\cdots+y_Q)+n_{s+i}(z_1+\cdots+z_R)\bigr|\leq mXM\]
for each $i=1,\dots,s$.
Taking $m=(\log X)^{1/2}$ with $X\geq \exp\bigl(Cd^2\log^2(\VE^{-1}sM^2)\bigr)$ for some suitably large absolute constant $C>0$ gives the result (since $N:=mXM$ can certainly then be dominated by $X^2$). \qed

\section{Simultaneous Polynomial Recurrence}\label{Sec3}

In this section we extend the elementary argument presented in Section \ref{Sec2} above to cover simultaneous polynomial recurrence modulo 1 for general systems of polynomials with real coefficients and no constant term. This approach again gives effective quantitative bounds with respect to the number of polynomials in question. Our main result is the following
 
 \begin{thm}\label{3}
Let $k$ be positive integer. Given any system of polynomials $h_1,\dots,h_d$ of degree at most $k$ with real coefficients and no constant term and an integer $N\geq1$, there exists an integer $1\leq n\leq N$ such that 
\[\|h_j(n)\|\ll \exp\bigl(-(4^kk!D_kd^{5k})^{-1}(\log N)^{1/D_k}\bigr) \text{ \ for \ }j=1,\dots,d\]
where $D_k\geq1$ is some absolute constant that depend only on $k$.
\end{thm}

\subsection{Reduction to the case of $k^{\text{th}}$ powers}\label{sim}
In this subsection we will show that Theorem \ref{3}, in the generality stated above, is in fact a rather straightforward consequence of the following natural intermediate generalization of Theorem \ref{1}.

\begin{thm}\label{2}
Let $k$ be positive integer. Given any real numbers $\A_1,\dots,\A_d$, and an integer $N\geq1$, there exists an integer $1\leq n\leq N$ such that \[\|n^k\A_j\|\ll \exp\bigl(-(C_kd^{5})^{-1}(\log N)^{1/C_k}\bigr) \text{ \ for \ }j=1,\dots,d\]
where $C_k\geq 1$ is some absolute constant depending only on $k$.
\end{thm}

\begin{proof}[Proof that Theorem \ref{2} implies Theorem \ref{3}]  This argument will follow along the same lines as the deduction of Corollary \ref{cor} from Theorem \ref{1} and Kronecker's theorem.
We will proceed by induction on $k$ and show that if, for arbitrary integers $k\geq2$ and $N\geq1$, the conclusion of Theorem \ref{2} holds for any given any real numbers $\A_1,\dots,\A_d$ (with this value of $k$) and that the conclusion of Theorem \ref{3} holds for any collection of $d$ polynomials of degree at most $k-1$ with real coefficients and no constant term, then the conclusion of Theorem \ref{3} in fact holds for any collection of $d$ polynomials of degree at most $k$ with real coefficients and no constant term.

Let $N\geq1$ be an integer and $h_1,\dots,h_d$ be any collection of $d$ polynomials of degree at most $k$ with real coefficients and no constant term. For each $1\leq j\leq d$ we shall express these as
\[h_j(x)=\A_jx^k + h'_j(x)\]
where $h'_j$ is a polynomial of degree at most $k-1$ with real coefficients and no constant term. 
We now set
\[N':=\exp\bigl((4kC_kd^{5})^{-D_{k-1}}(\log N)^{1/C_k}\bigr),\] 
noting that although slightly complicated this quantity is definitely smaller than $N^{1/4}$. 
The assumed validity of Theorem \ref{2} for $k$ allows us to conclude the existence of an integer $1\leq n\leq N/N'$ such that
\be\label{1/2}\|n^k\A_j\|\ll \exp\bigl(-(C_kd^{5})^{-1}(\log N/N')^{1/C_k}\bigr)\leq \exp\bigl(-(2C_kd^{5})^{-1}(\log N)^{1/C_k}\bigr)\ee
for $j=1,\dots,d$, and hence for \underline{all} $1\leq t\leq N'$ we have
\be\label{noD}\|(tn)^k\A_j\|\ll N'^k\exp\bigl(-(2C_kd^{5})^{-1}(\log N)^{1/C_k}\bigr)\leq \exp\bigl(-(4C_kd^{5})^{-1}(\log N)^{1/C_k}\bigr)\ee
for $j=1,\dots,d$. Note that for the last inequality in (\ref{1/2}) we simply used the fact that $N'\leq N^{1/2}$, while for the last inequality in (\ref{noD}) we used the fact that $N'\leq\exp\bigl((4kC_kd^{5})^{-1}(\log N)^{1/C_k}\bigr)$.

Finally we note that the inductive hypothesis, applied to the polynomials $H_1,\dots,H_d$ given by \[H_j(t)=h'_j(tn)\] for $j=1,\dots,d$, 
gives the existence of an integer $1\leq t\leq N'$ such that 
\begin{align*}\|h'_j(tn)\|&\ll \exp\bigl(-(4^{k-1}(k-1)!D_{k-1}d^{5(k-1)})^{-1}(\log N')^{1/D_{k-1}}\bigr)\\
&\leq \exp\bigl(-(4^{k}k!D_kd^{5k})^{-1}(\log N)^{1/D_k}\bigr) 
\end{align*}
for $j=1,\dots,d$, where $D_k=C_kD_{k-1}$ since by the definition of $N'$ we have
\[(\log N')^{1/D_{k-1}}=(4kC_kd^5)^{-1}(\log N)^{1/C_kD_{k-1}}.\qedhere\]
\end{proof}

\subsection{Proof of Theorem \ref{2}}
We again proceed by induction on $k$. We suppose that Theorem \ref{2} holds for all positive integers $1\leq k\leq K$, and hence also, in light of the arguments in Section \ref{sim} above, that Theorem \ref{3} holds for $K$.  We seek to verify Theorem \ref{2} for $k=K+1$.

For this fixed integer $K+1$, any given real parameter $\VE>0$ and positive integers $d$, $s$, and $M$, we let $P_{d,K+1}(s,M,\VE)$ denote the statement:
\begin{quote}
\emph{Given any real numbers $\A_1,\dots,\A_d$, there exists integers \mbox{$0\leq n_1,\dots,n_s\leq M$} (not all equal to zero) such that \[\|(n_1^{K+1}+\cdots+n_s^{K+1})\A_j\|\leq\VE  \text{ \ for \ }j=1,\dots,d.\]}
\end{quote}

As in the proof of Theorem \ref{1}, the verification of Theorem \ref{2} for $k=K+1$ relies on the the following key proposition:

\begin{propn}\label{Propn2}
Given any real parameter $\VE>0$ and positive integers $d$, $s$, and $M$, then
\[P_{d,K+1}(2s,M,\VE) \ \Longrightarrow \ P_{d,K+1}(s,N,2\VE)\]
for any integer 
\[N\geq \exp\bigl(C_1\bigl(d^2\log(\VE^{-1}sM^{K+1})\bigr)^{C_2}\bigr),\] 
where $C_1$ and $C_2$ are absolute constants greater than or equal to 1 that depend only on $K+1$.
\end{propn}

\begin{proof}[Proof that Proposition \ref{Propn2} implies Theorem \ref{2}]

In this case it follows immediately from Kronecker's theorem and any solution to \emph{Waring's Problem}\footnote{\,see \cite{P2} and Chapter 5 in \cite{P} 
for a discussion of elementary proofs of the Hilbert-Waring theorem.} that the statement $P_{d,K+1}(s, \VE^{-d/K+1}, \VE)$ holds provided $s=s(K+1)$ is chosen sufficiently large with respect to $K+1$. We will assume that $s$ is a power of 2. By applying Propostion \ref{Propn2} $\log_2s$ times we can thus conclude that the statement 
\[P_{d,K+1}(1,\underbrace{\exp\bigl(((K+1)C_1)^{\log_2s}\bigl(d^{4}\log(\VE^{-(d+1)}s)\bigr)^{C_2^{\log_2s}}\bigr)}_{(\star)}, s\VE)\]
also holds. Theorem \ref{2} follows as stated by setting $N=\text{($\star$)}$ and solving for $\VE$ as this gives
\[\VE=c\exp\bigl(-d^{-5}\bigl(((K+1)C_1)^{-\log_2s}\log N\bigr)^{1/C_2^{\log_2s}}\bigr).\qedhere\] \end{proof}

Matters thus reduce to verifying Proposition \ref{Propn2}.

\subsection{Proof of Proposition \ref{Propn2}}

Let $\A_1,\dots,\A_d$ be any given real numbers and $(m,X)$ be a pair of positive integer parameters to be determined. 

It follows from Theorem \ref{2}, via our inductive hypothesis with ``$d=d(m-1)^{K}$", that there exist integers $1\leq x_1,\dots,x_m\leq X$ such that we simultaneously have
\begin{align}\label{est1'}
\|x^k_{i_1}(x_{i_2}\dots x_{i_{K+2-k}}\A_j)\|
\ll \exp\bigl(-(4^KK!D_Kd^{3K}(m-1)^{3K^2})^{-1}(\log X)^{1/D_K}\bigr)
\end{align}
for all $1\leq i_1<i_2,\dots,i_{K+1}\leq m$, $j=1,\dots,d$ and $1\leq k\leq K$.

The pigeonhole principle implies that there exist two disjoint subsets $\{y_1,\dots,y_Q\}$ and $\{z_1,\dots,z_R\}$ of $\{x_1,\dots,x_m\}$ and real numbers $\B_1,\dots,\B_d$ such that for each $j=1,\dots,d$, the sums of the respective $(K+1)^{\text{th}}$ powers, namely \[Y=y_1^{K+1}+\cdots+y_Q^{K+1}\quad \text{and}\quad Z=z_1^{K+1}+\cdots+z_R^{K+1},\] each satisfy 
\[Y\A_j\equiv\B_j+\D_{Y,j} \imod{1} \quad\text{and}\quad Z\A_j\equiv\B_j+\D_{Z,j} \imod{1}\]
with both $|\D_{Y,j}|\leq 2^{-m/d}$ and $|\D_{Z,j}|\leq 2^{-m/d}$.

Now the assumed validity of the statement $P_{d,K+1}(2s,M,\VE)$ allows us to conclude that  there must exist integers $1\leq n_1,\dots,n_{2s}\leq M$ such that for each $j=1,\dots,d$ 
\[\|(n_1^{K+1}+\cdots+n_{2s}^{K+1})\B_j\|\leq\VE\]
and hence
\be\label{est2'}\|\{Y(n_1^{K+1}+\cdots+n_s^{K+1})+Z(n_{s+1}^{K+1}+\cdots+n_{2s}^{K+1})\}\A_j\|\leq\VE+2s M^{K+1} 2^{-m/d}.\ee

Since each $Y$ and $Z$ are themselves sums of $(K+1)^{\text{th}}$ powers (in fact sums of different $x_i^{K+1}$) and \[X_1^{K+1}+\cdots+X_n^{K+1}=(X_1+\cdots+X_n)^{K+1} - \!\!\!\!\!\!\!\!\!\!\!\!\sum_{\substack{i_1,\dots,i_{K+1}=1\\ \text{diagonal excluded}}}^n \!\!\!\!\!\!\!\!\!\! X_{i_1}\cdots X_{i_{K+1}},\]
it follows from (\ref{est1'}) and (\ref{est2'}) that
\begin{align*}
&\left\|\left(\sum_{i=1}^s \Bigl\{n_i(y_1+\cdots+y_Q)+n_{s+i}(z_1+\cdots+z_R)\Bigr\}^{K+1}\right)\A_j\right\|  \\
&\quad\quad
\leq \VE+2s M^{K+1} 2^{-m/d}+O\bigl(sm^{K+1} M^{K+1}  \exp(-(4^KK!D_Kd^{3K}(m-1)^{3K^2})^{-1}(\log X)^{1/D_K})\bigr)
\end{align*}
for each $j=1,\dots,d$. Note also that 
\[\bigl|n_i(y_1+\cdots+y_Q)+n_{s+i}(z_1+\cdots+z_R)\bigr|\leq mXM\]
for each $i=1,\dots,s$.
Taking 
\[m=\bigl((4^{K}K!D_Kd^{3K-1})^{-1}(\log N)^{1/D_{K}}\bigr)^{1/3K^2+1}\] 
with
\[X\geq \exp\bigl(C\bigl((4^{K}K!D_k)\bigl(d^{2}\log (\VE^{-1}sM^{K+1})\bigr)^{3K^2+1}\bigr)^{D_{K}}\bigr)\]
 for some suitably large absolute constant $C>0$ gives the result, since in this case the quantity $N:=mXM$ can certainly then be dominated by $X^2$.
\qed

\end{document}